\documentclass[12pt]{l4dc2024}
\usepackage{amsmath}
\usepackage{amssymb}
\usepackage{stackengine}
\usepackage{mathrsfs}
\usepackage{bbm}
\usepackage{mathtools}
\usepackage{bm}
\usepackage{amsfonts}
\usepackage{cases}
\usepackage{wrapfig}
\usepackage{graphicx}
\usepackage{float}
\usepackage{url}
\usepackage{multirow}
\usepackage{color}
\usepackage{algorithm}
\usepackage{algorithmic}
\usepackage{graphicx}
\usepackage{epstopdf}
\usepackage{lipsum}
\usepackage{booktabs}
\usepackage{threeparttable}
\usepackage{url}
\usepackage{hyperref}
\usepackage{xcolor}
\definecolor{mycolorc1}{RGB}{255, 63, 164}
\definecolor{mycolorc2}{RGB}{255, 163, 60}
\definecolor{mycolorc3}{RGB}{177, 94, 255}
\definecolor{mycolorc4}{RGB}{61, 48, 162}
\newtheorem{assumption}{Assumption}
\newcommand{\argmin}{\mathop{\rm argmin}}

\title[Learning-based Rigid Tube MPC]{Learning-based Rigid Tube Model Predictive Control}
\author{
\Name{Yulong Gao}$^1$ \Email{yulong.gao@imperial.ac.uk} \\
\Name{Shuhao Yan}$^2$ \Email{shuhao.yan@mathematik.uni-stuttgart.de} \\
\Name{Jian Zhou}$^3$ \Email{jian.zhou@liu.se} \\
\Name{Mark Cannon}$^4$ \Email{mark.cannon@eng.ox.ac.uk} \\
\Name{Alessandro Abate}$^5$ \Email{alessandro.abate@cs.ox.ac.uk}\\
\Name{Karl H. Johansson}$^6$ \Email{kallej@kth.se} \\
\addr $^1$  Department of Electrical and Electronic Engineering, Imperial College London, UK\\
\addr $^2$ Department of Mathematics, University of Stuttgart, Germany \\
\addr $^3$ Department of Electrical Engineering, Link\"oping University, Sweden\\
\addr $^4$ Department of Engineering Science, University of Oxford, UK\\
\addr $^5$  Department of Computer Science, University of Oxford, UK\\
\addr $^6$ Division of Decision and Control Systems, KTH Royal Institute of Technology, Sweden
}

\begin{document}

\maketitle
\vspace{-1cm}
\begin{abstract}
 This paper is concerned with model predictive control (MPC) of discrete-time linear systems subject to bounded additive disturbance and mixed constraints on the state and input, whereas the true disturbance set is unknown. Unlike most existing work on robust MPC, we propose an algorithm incorporating online learning that builds on prior knowledge of the disturbance, i.e., a known but conservative disturbance set. We approximate the true disturbance set at each time step with a parameterised set, which is referred to as a quantified disturbance set, using disturbance realisations. A key novelty is that the parameterisation of these quantified disturbance sets enjoys desirable properties such that the quantified disturbance set and its corresponding rigid tube bounding disturbance propagation can be efficiently updated online. We provide statistical gaps between the true and quantified disturbance sets, based on which, probabilistic recursive feasibility of MPC optimisation problems is discussed. Numerical simulations are provided to demonstrate the effectiveness of our proposed algorithm and compare with conventional robust MPC algorithms.
\end{abstract}

\begin{keywords}
  Rigid tube MPC, Learning uncertainty, Scenario approach.
\end{keywords}

\begin{sloppypar}
\section{Introduction}\label{sec:intro}
Robust model predictive control (MPC) \citep{goodwin2014robust} can ensure constraint satisfaction by design under all realisations of uncertainty. This guarantee is essential in many practical applications, such as safety critical systems and power system operation. Tube-based methods \citep{langson2004robust,mayne2011robust} are one of the most well-known approaches in robust MPC.  Relying on knowledge of the worst-case disturbance bounds, one can calculate a sequence of sets to bound disturbance propagation over an infinite horizon. This sequence is referred to as a tube. Nevertheless, tube-based methods suffer from potential conservatism, as the disturbance sets used in these worst-case strategies can be overly conservative, resulting in small feasible sets \citep{saltik2018outlook}. 
The tube is usually either constructed  offline, or is otherwise subject to constraints based on prior  information on model uncertainty, and it can therefore be difficult to exploit information obtained during online implementation
to refine the uncertainty model and update the tube accordingly. 
This precludes the use of data-driven methods, i.e., the scenario approach \citep{calafiore2006scenario} and sample average approximation \citep{doi:10.1137/070702928}, to further quantify uncertainty on the basis of online measurements.

\begin{wrapfigure}{r}{0.3\textwidth}
    \centering
    \includegraphics[width=0.3\textwidth]{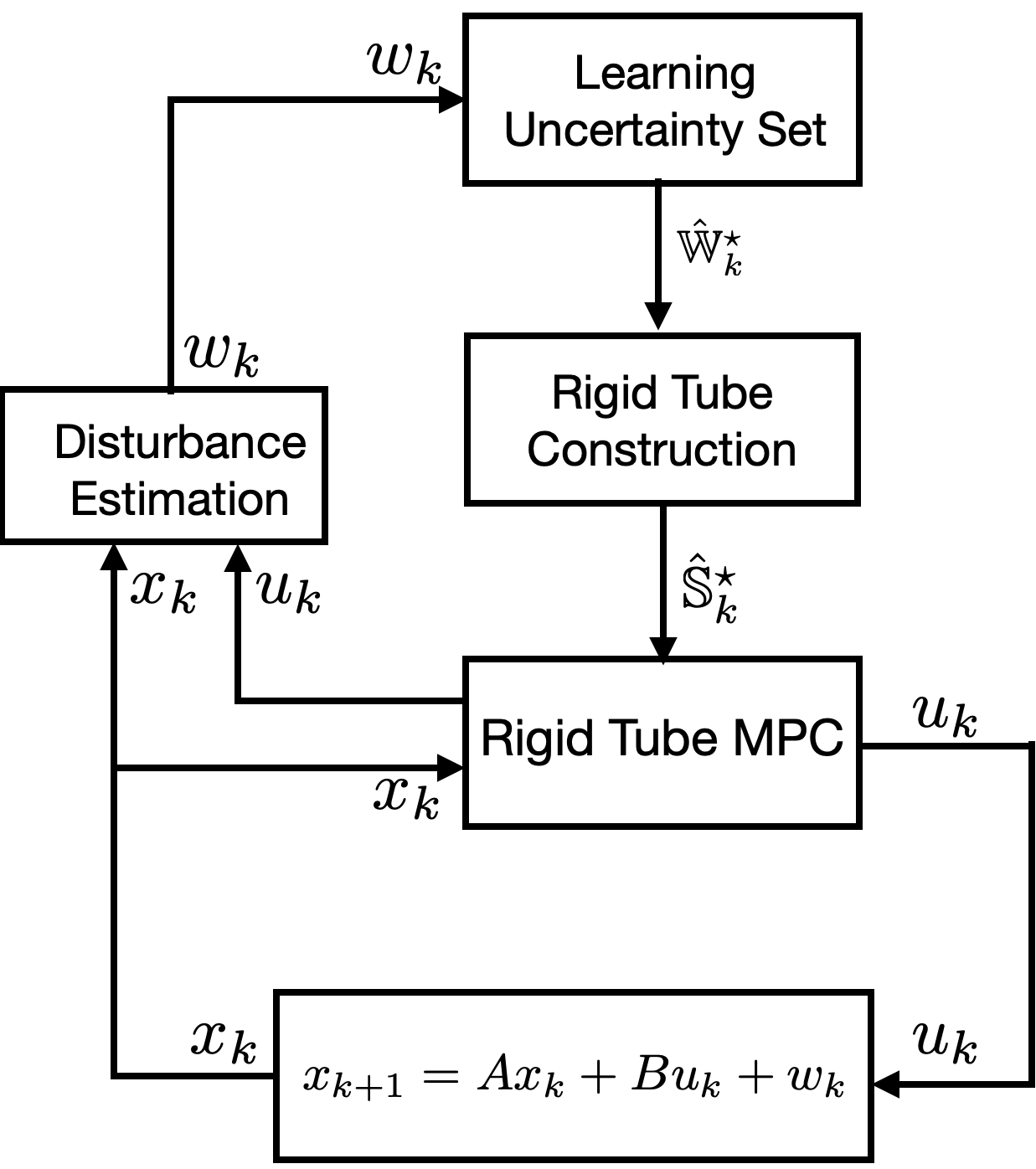}
    \caption{\small The framework.}
    \label{Fig:HILP}
\end{wrapfigure}
To mitigate the aforementioned conservatism, we propose an MPC algorithm incorporating online learning of the exact yet unknown disturbance set $\mathbb{W}_{\rm true}$. The framework is shown in Fig.~\ref{Fig:HILP}. Given a known yet conservative disturbance set $\mathbb{W}$ ($\supseteq\mathbb{W}_{\rm true}$) and disturbance realisations, we construct a set $\hat{\mathbb{W}}^{\star}_k$ as a homothetic transformation of $\mathbb{W}$ to approximate $\mathbb{W}_{\rm true}$ at each time step. Accordingly, the MPC optimisation is updated at each time step. The main contributions of this paper are: 1) The parameterisation of set  $\hat{\mathbb{W}}^{\star}_k$ enables the scenario programs yielding the problem to be equivalently rewritten as linear programs (LPs). 2) We provide statistical gaps between $\mathbb{W}_{\rm true}$ and $\hat{\mathbb{W}}^{\star}_k$. The latter leads to larger feasible sets of initial conditions than that of the tube MPC based on $\mathbb{W}$. 3) The rigid tube $\{\hat{\mathbb{S}}^{\star}_k,\hat{\mathbb{S}}^{\star}_k,\ldots\}$  based on the set $\hat{\mathbb{W}}^{\star}_k$ is updated efficiently online, without resorting to recomputing an outer approximation of the Minkowski sum of infinitely many polytopic sets.

\paragraph*{Related work} Recent developments in data-driven methods have considerably improved our ability to infer model uncertainty from measurements. In chance-constrained problems, a typical approach is to identity a high-probability region of the uncertainty space using samples.
This region is then used to determine robust constraints that conservatively replace the original chance constraints.
Different parameterisations of the region have been proposed, including hyper-rectangles \citep{6727399}, norm balls \citep{9206383} and a finite union of norm balls \citep{9303863}. Similar ideas are applied to uncertainty in stochastic MPC problems \citep{8727721,shang2019data}, where support vector clustering is used to identify a high-density region of disturbance. Conformal prediction \citep{shafer2008tutorial} is another technique for uncertainty quantification. For example, in \cite{dixit2023adaptive,lindemann2023safe}, it is used to quantify uncertainty in predicted agent trajectories and to provide high-confidence regions surrounding the predicted trajectories using past observations. These regions are then employed in an MPC problem, which is solved to perform motion planning with probabilistic safety guarantees.

In the literature of robust and stochastic MPC problems, the scenario approach is often used to handle uncertainty. A typical application is to approximate a probabilistic constraint with a number of hard constraints, see, e.g., \cite{6426462,SCHILDBACH20143009}. This approximation results in significant computational challenges. To avoid this issue, \citet{8882241} exploits probabilistic reachable sets which are constructed offline using the scenario approach. Although this work provides closed-loop constraint satisfaction, it uses an MPC initialisation scheme \citep{HEWING2020109095}, in which there is no direct feedback from true state measurements. It is yet unclear whether this setting leads to adequate closed-loop performance in general \citep{MAYNE2018169}.

\paragraph*{Notation} $\mathbb{N}$ is the set of nonnegative integers. $\mathbb{N}_{\geq n}$ is the set of nonnegative integers that are larger than or equal to $n$. $\mathbb{R}$ is the set of real numbers.  For sets $\mathbb{X}$ and $\mathbb{Y}$, $\mathbb{X}\oplus \mathbb{Y}=\{x+y\mid x\in\mathbb{X}, y\in \mathbb{Y}\}$. 
Matrices of appropriate dimension with all elements equal to 1 and 0 are denoted by $\bm{1}$ and $\bm{0}$, respectively. A positive (semi-) definite matrix $Q$ is denoted by $Q\succ 0$ ($Q\succcurlyeq 0$).  The cardinality of a set $\mathbb{X}$ is denoted by $|\mathbb{X}|$.


\section{Problem Description}\label{Sec:Preliminaries}


We consider an uncertain discrete-time linear system described by
\begin{eqnarray}\label{LTIsystem}
x_{k+1} = Ax_{k}+Bu_{k}+w_k, 
\end{eqnarray}
which is subject to a mixed constraint $Fx_k+Gu_k\leq \bm{1}$ with $F\in \mathbb{R}^{n_c \times n_x}$ and $G\in  \mathbb{R}^{n_c \times n_u}$. In \eqref{LTIsystem}, $x_k \in \mathbb{R}^{n_{x}}$ is the state,  $u_k\in \mathbb{R}^{n_{u}}$ the control input, and $w_k\in\mathbb{W}_{\rm true}\subset\mathbb{R}^{n_{x}}$ the additive disturbance. We assume perfect state feedback and that the matrices $A$, $B$ are known, while the \emph{true disturbance set} $\mathbb{W}_{\rm true}$ is unknown. 

\begin{assumption}\label{Ass}
The pair $(A,B)$ is stabilisable, and set $\mathbb{W}_{\rm true}$ is a convex subset of a known convex and compact polytope $\mathbb{W}$ in the form of $\mathbb{W}=\{w\in \mathbb{R}^{n_x}\mid  V w\leq \bm{1}\}$, where $V\in \mathbb{R}^{n_v\times n_x}$. The disturbance realisations $w_k$ of the system~\eqref{LTIsystem} are independent and identically distributed (i.i.d.) according to an unknown probability distribution ${\rm Pr}$ with support $\mathbb{W}_{\rm true}$. 
\end{assumption}

\subsection{Rigid Tube MPC based on Conservative Disturbance Set $\mathbb{W}$}
We first review the theory of rigid tube MPC. 
Given a prediction horizon $N$, one can decompose the predicted dynamics at time step $k$ as
\begin{equation} \label{Eq:predyn}
x_{i|k} = s_{i|k}+e_{i|k},  \ u_{i|k} = Kx_{i|k}+c_{i|k}, \ s_{i+1|k} = \Phi s_{i|k}+Bc_{i|k}, \ e_{i+1|k} = \Phi e_{i|k}+w_{i|k},
\end{equation}
where $\Phi=A+BK$ and $w_{i|k}\in \mathbb{W}$ for all $i\in\mathbb{N}$. In \eqref{Eq:predyn}, $s_{i|k}$ and $e_{i|k}$ are nominal and uncertain components of the state, respectively. The feedback gain $K\in \mathbb{R}^{n_u \times n_x}$ is fixed, while the free variable is $\bm{c}_k=[c^T_{0|k} \ \cdots \ c^T_{N-1|k}]^T$ with $c_{i|k}=\bm{0}$ for $i\in \mathbb{N}_{\geq N}$.

Given the state decomposition \eqref{Eq:predyn}, we consider a nominal predicted cost as the objective of online MPC optimisation at time step $k$ given by 
\begin{align}\label{Eq:nomcost}
J(s_{0|k}, \bm{c}_k)=\sum_{i=0}^{\infty}(\|s_{i|k}\|^2_{Q}+\|v_{i|k}\|^2_{R}),
\end{align}
where $Q\succcurlyeq 0$, $R\succ 0$, and $v_{i|k}=Ks_{i|k}+c_{i|k}$.
Suppose the  matrix pair $(A,Q)$ is observable. Then, there exists a unique solution, $P_x (\succ 0)$, to the following algebraic Riccati equation
\begin{align*}
P_x=A^TP_xA+Q-A^TPB(B^TP_xB+R)^{-1}B^TP_xA.
\end{align*}
Let $K=-(B^TP_xB+R)^{-1}B^TP_xA$, which ensures that $\Phi$ is strictly stabilising. As shown in \cite{kouvaritakis2016model}, \eqref{Eq:nomcost} can be rewritten in a compact form as
$$
J(s_{0|k}, \bm{c}_k)=\|s_{0|k}\|^2_{P_x}+\|\bm{c}_k\|^2_{P_c}, ~\text{with} \ P_c={\rm diag}\{B^TP_xB+R, \cdots, B^TP_xB+R\}.
$$

\begin{lemma}[{\cite{rakovic05}}]\label{Lemma:tubeforW}
If  Assumption~\ref{Ass} holds and $\Phi$ is strictly stable, there exist a finite integer $r$ and a scalar $\rho\in [0,1)$ such that  (i) $\Phi^r\mathbb{W}\subseteq \rho \mathbb{W}$
and (ii) the set 
\begin{align}\label{Eq:Wtube}
    \mathbb{S}=\frac{1}{1-\rho} \bigoplus_{i=0}^{r-1}\Phi^i\mathbb{W}
\end{align} is a convex and compact set,  
satisfying $\Phi \mathbb{S} \oplus \mathbb{W}\subseteq \mathbb{S}$ and $\bigoplus_{i=0}^{\infty}\Phi^i\mathbb{W}\subseteq\mathbb{S}$.
\end{lemma}

The sequence $\{\mathbb{S},\mathbb{S},\ldots\}$ provides a bound on $\{e_{0|k},e_{1|k},\ldots\}$, and is referred to as a \emph{rigid tube}, which further helps to define a vector of constraint tightening parameters as 
\begin{align}\label{Eq:hs}
h_s=\max_{e\in \mathbb{S}} (F+GK)e.
\end{align}
In \eqref{Eq:hs}, the maximisation is performed for each row of $F+GK$. This is then used to reformulate the mixed constraint as a deterministic constraint. Now we can formulate the rigid tube MPC optimisation problem to be solved at time step $k$ as
\begin{equation}\label{Eq:RigTubeMPCOPT}
{\rm OPT}(\mathbb{S}, h_s,\nu_s): 	\begin{cases}
\min\limits_{s_{0|k},\bm{c}_k}\quad \|s_{0|k}\|^2_{P_x}+\|\bm{c}_k\|^2_{P_c}   \\
{\rm s.t}: x_k-s_{0|k}\in \mathbb{S}, \
\bar{F}\Psi^i \!\begin{bmatrix}
        s_{0|k} \\
        \bm{c}_k
    \end{bmatrix}\!\!\leq\! \bm{1}-h_s, 	 \forall i\in\mathbb{N}_{[0,\nu_s]},
\end{cases}
\end{equation}
where $\bar{F}=\begin{bmatrix}
F+GK & GE
\end{bmatrix}$, $\Psi=\begin{bmatrix}
\Phi & BE \\ \
\bm{0} & M
\end{bmatrix}$,  $E=\begin{bmatrix}
I_{n_u} & \bm{0} & \cdots & \bm{0}
\end{bmatrix}$, and $M$ is the block-upshift operator. Problem \eqref{Eq:RigTubeMPCOPT} is necessarily recursively feasible if $\nu_s$ is chosen as the smallest positive integer such that $\bar{F}\Psi^{\nu_s+1}z \leq \mathbf{1}-h_s$ for all $z$ satisfying $\bar{F}\Psi^i z\leq \mathbf{1}-h_s$, $i\in \mathbb{N}_{[0,\nu_s]}$.

\subsection{Problem under Study}
In practice, it is challenging to determine exact bounds on the disturbance $w_k$. To ensure robustness, we usually tailor a conservative set $\mathbb{W}$ to contain all possible realisations of $w_k$. Such a set, however, may lead to a small feasible region or even infeasibility of problem~\eqref{Eq:RigTubeMPCOPT}. Under Assumption~\ref{Ass}, $\mathbb{W}$ is an \emph{a priori} known yet conservative set for $w_k$, whereas
$\mathbb{W}_{\rm true}$
is an exact yet unknown set that tightly contains all possible realisations of $w_k$, and we have $\mathbb{W}_{\rm true} \subseteq \mathbb{W}$. These conditions motivate us to learn the set  $\mathbb{W}_{\rm true}$ and then use the learned disturbance set, denoted by $\hat{\mathbb{W}}_k^{\star}$, to reduce the conservativeness of the rigid tube formulation.
More specifically, our aim is two-fold: (1)  characterisation of the gap between $\mathbb{W}_{\rm true}$ and $\hat{\mathbb{W}}_k^\star$; (2) analysis of the rigid tube MPC algorithm based on $\hat{\mathbb{W}}_k^\star$.

\section{Learning Uncertainty Set}\label{section:Uncertainty Quantification}
In this section, we approximate the  set $\mathbb{W}_{\rm true}$ online using the set $\mathbb{W}$  and the collected disturbance realisations. We first define the initial disturbance information set that contains the disturbance samples collected offline as $\mathcal{I}^w_{0}=\{w^s_{i}, \ i=-N_{\rm off}, \cdots, -1\}$, where $N_{\rm off}\in\mathbb{N}$. As states can be measured exactly, we can recover a new disturbance sample at  time step $k+1$ by $w^s_{\rm new}=x_{k+1}-Ax_k-Bu_k$ and update the  disturbance information set as 
\begin{align}\label{Eq:Iwdynamics}		\mathcal{I}^w_{k+1}=\mathcal{I}^w_{k}\cup \{w^s_{\rm new}\}. 
\end{align}

It follows that $\mathbb{W}_{\rm true}$ is approximated by computing the minimum set in the form of 
\begin{align}\label{Eq:Wva}
\mathcal{W}(v_k,\alpha_k):=(1-\alpha_k)v_k\oplus \alpha_k \mathbb{W},
\end{align} 
which contains all samples in $\mathcal{I}^w_{k}$. The design parameters are $v_k\in \mathbb{W}$ and $\alpha_k\in [0,1]$. 
Suppose Assumption~\ref{Ass} holds, it is shown in \cite{gao2021invariant} that $\mathcal{W}(v,\alpha)$ defined in \eqref{Eq:Wva} satisfies: (i) $\mathcal{W}(v,\alpha)\subseteq \mathbb{W}$, for all $v\in \mathbb{W}$ and $\alpha\in [0,1]$; (ii) $\mathcal{W}(v,\alpha_1)\subseteq \mathcal{W}(v,\alpha_2)$, for all $v\in \mathbb{W}$ and $
0\leq \alpha_1\leq \alpha_2\leq 1$. 
Given these properties, we formulate the following optimisation problem to minimise the set $\mathcal{W}(v_k,\alpha_k)$ while containing all samples in $\mathcal{I}^w_{k}$: 
\begin{eqnarray}	
\begin{cases}
    &\min\limits_{v,\alpha}\quad  \alpha   \\
    & \hspace{0.1cm}{\rm s.t}: w^s\in \mathcal{W}(v,\alpha), \forall w^s\in \mathcal{I}^w_{k}, \\ & \hspace{0.8cm} v\in \mathbb{W}, \	\alpha\in [0,1],
\end{cases} \Longleftrightarrow 
\begin{cases}
    &\min\limits_{v,\alpha}\quad  \alpha   \\
    & \hspace{0.1cm}{\rm s.t}: -(1-\alpha)V v\leq \alpha \bm{1}-Vw^s, \forall w^s\in \mathcal{I}^w_{k}, \\ 
    & \hspace{0.8cm} Vv\leq \bm{1}, \ \alpha\in [0,1]. 
    \end{cases}\label{Opt2:quanset}
\end{eqnarray}  

Proposition~\ref{prop: LP equivalence} shows that the nonconvex problem \eqref{Opt2:quanset} can be reformulated as an LP. 
\begin{proposition}\label{prop: LP equivalence}
Under Assumption~\ref{Ass}, the optimal solution to the problem \eqref{Opt2:quanset}  can be obtained by solving the following LP:
\begin{eqnarray}	\label{Opt3:quansetLP}
   \max\limits_{y,\beta} ~\beta   \ \ 
       {\rm s.t}: -V y\leq (1-\beta)\bm{1}-Vw^s, \forall w^s\in \mathcal{I}^w_{k}, \ Vy \leq \beta\bm{1}, \ 	\beta\in [0,1]. 
\end{eqnarray}  
\end{proposition}
\begin{proof}
Replacing $\alpha$ and $v$ in \eqref{Opt2:quanset} with $\beta = 1-\alpha$ and  $y =\beta v$ directly yields \eqref{Opt3:quansetLP}.
\end{proof}

Denote by $(y^{\star}_k,\beta^{\star}_k)$  the optimiser to \eqref{Opt3:quansetLP}. Let $\alpha_k^{\star}=1-\beta^{\star}_k$ and $v_k^{\star}=y^{\star}_k/\beta^{\star}_k$ if $\beta^{\star}_k>0$. Then the quantified disturbance set at time step $k$ is given by
\begin{eqnarray}  \label{Eq:optWtrue}
\hat{\mathbb{W}}^{\star}_k=	\mathcal{W}(v^{\star}_k,\alpha_k^{\star}).
\end{eqnarray}  
Note that when $\beta^{\star}_k=0$ and $\alpha_k^{\star}=1$, $\hat{\mathbb{W}}^{\star}_k=\mathbb{W}$ regardless of the choice of $v_k\in \mathbb{W}$. Next we provide a statistical gap between $\hat{\mathbb{W}}^{\star}_k$ and $\mathbb{W}_{\rm true}$. 

\begin{theorem}\label{Theo: risk bound}
Suppose Assumption  \ref{Ass} holds. Given   $\epsilon, \gamma \in (0,1)$ and Euler’s  number $\rm{e}$, if $|\mathcal{I}^w_k|\geq \frac{1}{\epsilon} \frac{\rm{e}}{\rm{e}-1} \Bigl(\ln\frac{1}{\gamma} + n_x \Bigr)$, then ${\rm Pr}[w^r\in \mathbb{W}_{\rm true}: w^r\notin \hat{\mathbb{W}}^{\star}_k]\leq  \epsilon$ is satisfied with probability no less than $1-\gamma$. 
\end{theorem}
\begin{proof}
We first define the following robust LP:
\begin{eqnarray*}	\label{Opt3:robustquansetLP}
        &&\max\limits_{y,\beta}~ \beta   \ \ \hspace{0.1cm}{\rm s.t}: -V y\leq (1-\beta)\bm{1}-Vw^s, \forall w^s\in \mathbb{W}_{\rm true}, \ Vy \leq \beta\bm{1}, \ \beta\in [0,1]. 
\end{eqnarray*} 
Problem \eqref{Opt3:quansetLP} is its corresponding scenario LP, where  $\mathcal{I}^w_{k}$ is the set of i.i.d. samples from the convex uncertainty set $\mathbb{W}_{\rm true}$. Therefore, \cite[Theorem 4]{5531078} yields the stated sample complexity and the confidence guarantee in Theorem~\ref{Theo: risk bound}. 
\end{proof}

Given the prescribed $\epsilon$ and $\gamma$, one can determine the size of the initial disturbance information set $\mathcal{I}^w_{0}$ according to Theorem~\ref{Theo: risk bound}.

\subsection{Online Update of $\hat{\mathbb{W}}^{\star}_{k}$}
Note that the computation complexity of \eqref{Opt3:quansetLP} increases with the update of set $\mathcal{I}^w_{k}$. To mitigate this, we can online update the set 	$\hat{\mathbb{W}}^{\star}_{k}=	\mathcal{W}(v^{\star}_{k},\alpha_{k}^{\star})$ as follows. If the disturbance realisation $w_{k-1} (=x_{k}-Ax_{k-1}-Bu_{k-1})$ collected at time step $k$ is in the set $\hat{\mathbb{W}}^{\star}_{k-1}$, let $(v^{\star}_{k},\alpha_{k}^{\star})=(v^{\star}_{k-1},\alpha_{k-1}^{\star})$; otherwise, we solve the following problem:
\begin{eqnarray}	\label{Opt:onlinequanset}
(v^{\star}_{k},\alpha_{k}^{\star})=\argmin\limits_{v,\alpha}\{\alpha \ | \ \hat{\mathbb{W}}^{\star}_{k-1}\subseteq	\mathcal{W}(v,\alpha),w_{k-1}\in \mathcal{W}(v,\alpha),v\in \mathbb{W}, \ \alpha\in [0,1]\} .
\end{eqnarray} 

\begin{proposition}
The optimal solution to \eqref{Opt:onlinequanset} can be obtained by solving the following LP:
\begin{equation}
\label{Opt:onlinequansetLP}
\begin{cases}
    &\max\limits_{y,\beta}\quad  \beta   \\
    & \hspace{0.1cm}{\rm s.t}:
    \begin{cases}
        -V y\leq (1-\beta-\alpha^{\star}_{k-1})\bm{1}-(1-\alpha^{\star}_{k-1})Vv_{k-1}^{\star}, \\ 
        -V y\leq (1-\beta)\bm{1}-Vw_{k-1}, \ Vy \leq \beta\bm{1}, \ \beta\in [0,1].
    \end{cases}
\end{cases}
\end{equation}	
\end{proposition}
\begin{proof}
Since $\mathcal{W}(\alpha,v)=\{z\in \mathbb{R}^{n_x}\mid Vz \leq \alpha \bm{1}+(1-\alpha)Vv\}$, it follows that $	\hat{\mathbb{W}}^{\star}_{k-1}\subseteq	\mathcal{W}(v,\alpha)$ holds if and only if 
$\alpha_{k-1}^{\star} \bm{1}+(1-\alpha^{\star}_{k-1})Vv^{\star}_{k-1}\leq \alpha \bm{1}+(1-\alpha)Vv$. Then  \eqref{Opt:onlinequansetLP} is obtained using the same argument as the proof of Proposition~\ref{prop: LP equivalence}.
\end{proof}

\vspace{-0.6cm}
\section{Learning-based Rigid Tube MPC}\label{Sec:MPC}
In this section, we formulate the tube MPC problem using the quantified disturbance set $\hat{\mathbb{W}}^{\star}_k$. 
We first show how to efficiently construct the rigid tube for  the set   $\hat{\mathbb{W}}^{\star}_k$. 

\begin{proposition}\label{Prop:Quantifiedtube}
Suppose that Assumption~\ref{Ass} holds. 	Given the sets $\hat{\mathbb{W}}^{\star}_k=\mathcal{W}(v^{\star}_k,\alpha_k^{\star})$  in \eqref{Eq:optWtrue} and  $\mathbb{S}$ in \eqref{Eq:Wtube}, we can construct a set that satisfies $\Phi\hat{\mathbb{S}}^{\star}_k\oplus \hat{\mathbb{W}}_k^{\star}\subseteq \hat{\mathbb{S}}^{\star}_k$ via
\begin{align}\label{Eq:quantube}
    \hat{\mathbb{S}}^{\star}_k=\alpha^{\star}_k \mathbb{S}\oplus (1-\alpha^{\star}_k)(I-\Phi)^{-1}v^{\star}_k.
\end{align}
\end{proposition}
\begin{proof}
From Lemma~\ref{Lemma:tubeforW}, we have $\Phi \mathbb{S} \oplus \mathbb{W}\subseteq \mathbb{S}$. 
Since $\Phi$ is strictly stable and $I-\Phi$ is invertible,  we have
\begin{align*}
\Phi\hat{\mathbb{S}}^{\star}_k\oplus \mathcal{W}(v^{\star}_k,\alpha_k^{\star})&=\alpha_k^{\star}\Phi \mathbb{S}\oplus(1-\alpha_k^{\star}) \Phi(I-\Phi)^{-1}v^{\star}_k \oplus \alpha_k^{\star} \mathbb{W}\oplus (1-\alpha_k^{\star})v^{\star}_k\\
    &= \alpha_k^{\star} (\Phi \mathbb{S}\oplus \mathbb{W})\oplus (1-\alpha_k^{\star})(\Phi(I-\Phi)^{-1}+I)v^{\star}_k\\
    &= \alpha_k^{\star} (\Phi \mathbb{S}\oplus \mathbb{W})\oplus (1-\alpha_k^{\star})(I-\Phi)^{-1}v^{\star}_k \subseteq  \alpha_k^{\star} \mathbb{S}\oplus (1-\alpha_k^{\star})(I-\Phi)^{-1}v^{\star}_k,
\end{align*}  
\noindent which yields the expression in \eqref{Eq:quantube}.
\end{proof}

Proposition~\ref{Prop:Quantifiedtube} implies that the  rigid tube $\{\hat{\mathbb{S}}^{\star}_k,\hat{\mathbb{S}}^{\star}_k,\ldots\}$ for the set $\hat{\mathbb{W}}^{\star}_k$ can be directly constructed using $\{\mathbb{S},\mathbb{S},\ldots\}$. This provides significant computational advantages since we do not need to recompute the rigid tube using \eqref{Eq:Wtube} when the set $\hat{\mathbb{W}}^{\star}_k$ is updated online. Based on  $\hat{\mathbb{S}}^{\star}_k$ in \eqref{Eq:quantube}, the corresponding tube MPC problem at time step $k$ can be formulated as
\begin{equation}\label{Eq:QuanRigTubeMPCOPT}
{\rm OPT}(\hat{\mathbb{S}}^{\star}_k,h_k^{\star},\nu_k): 	\begin{cases}
    \min\limits_{s_{0|k},\bm{c}_k}\quad \|s_{0|k}\|^2_{P_x}+\|\bm{c}_k\|^2_{P_c}   \\
    {\rm s.t}: x_k-s_{0|k}\in \hat{\mathbb{S}}^{\star}_k, \
        \bar{F}\Psi^i \begin{bmatrix}
            s_{0|k} \\
            \bm{c}_k
        \end{bmatrix}\leq \bm{1}-h^{\star}_k, 	\  \forall i\in\mathbb{N}_{[0,\nu_k]},
\end{cases}
\end{equation}
where $h^{\star}_k$ is defined as
\begin{equation}\label{Eq:hk}
h^{\star}_{k}=\max_{e\in\hat{ \mathbb{S}}_k^{\star}} (F+GK)e,    
\end{equation}
and $\nu_k$  is a  positive integer such that 
$\bar{F}\Psi^{\nu_k+1}z \leq \mathbf{1}-h^{\star}_k$ for all $z$ satisfying $\bar{F}\Psi^i z\leq \mathbf{1}-h^{\star}_k$, $i\in \mathbb{N}_{[0,\nu_k]}$. 

\subsection{Computing $\nu_k$}
We next consider how to efficiently compute $\nu_k$ in problem  \eqref{Eq:QuanRigTubeMPCOPT}. 
Let $\Omega(q,\nu)=\{z\in \mathbb{R}^{n_x+N n_u}\mid \bar{F}\Psi^{i}z \leq q, \
i \in \mathbb{N}_{[0,\nu]}
\}$, where $q\in  \mathbb{R}^{n_c}$ and $\nu\in \mathbb{N}$.  Let $\nu_s$ be such that 	\begin{equation}\label{Eq:nus}
\max_{z\in \Omega(1-h_s,\nu_s)}\bar{F}\Psi^{\nu_s+1}z \leq \bm{1}-h_s
\end{equation} 
and $P_s=P_s^T\in \mathbb{R}^{(n_x+N n_u)\times (n_x+N n_u)}$, $P_s\succ 0$ such that 
\begin{equation}\label{Eq:Ps}
\{z\in \mathbb{R}^{n_x+N n_u}\mid z^TP_sz\leq 1\} \supseteq \Omega(\bm{1}-h_s,\nu_s).
\end{equation}
At time step $k$, we want to find $\nu_k$ such that $\max_{z\in \Omega(1-h^{\star}_k,\nu_k)}\bar{F}\Psi^{\nu_k+1}z\leq \bm{1}-h_k^{\star}.$ Algorithm~\ref{Alg:Comnuk} provides a procedure to compute such $\nu_k$.  

\begin{algorithm}\footnotesize
\caption{Computation of $\nu_k$.} 
\begin{algorithmic}[1]
    \REQUIRE $\mathbb{S}$ in \eqref{Eq:Wtube},  $h_s$ in \eqref{Eq:hs},  $\nu_s$ in \eqref{Eq:nus}, and $P_s$ in \eqref{Eq:Ps}.
    \STATE Compute $h_k^{\star}$ in \eqref{Eq:hk}
    \STATE Let $\zeta_k=\max\limits_{i}[\bm{1}-h_k^*]_i/[\bm{1}-h_s]_{i}$ and $\nu=\nu_s$
    \WHILE{$\max\limits_{i} [\bar{F}]_i\Psi^{\nu+1}P_s^{-1}{\Psi^{\nu+1}}^T[\bar{F}]^T_i- [1-h^{\star}_k]_i/\zeta^2_k>0$} 
        {\STATE Let $\nu=\nu+1$}
    \ENDWHILE 
        \RETURN $\nu_k = \nu$
\end{algorithmic} \label{Alg:Comnuk} 
\end{algorithm}

\vspace{-0.5cm}
\begin{proposition}
If $h_s< \bm{1}$ in~(\ref{Eq:hs}),  
the integer $\nu_k$ obtained from Algorithm~\ref{Alg:Comnuk} necessarily satisfies 
$\max_{z\in \Omega(1-h^{\star}_k,\nu_k)}\bar{F}\Psi^{\nu_k+1}z\leq \bm{1}-h_k^{\star}$.
\end{proposition}
\begin{proof}
First, since $\bigoplus_{i=0}^{\infty}\Phi^i\mathbb{W}\subseteq\mathbb{S}$ and    $v^{\star}_k\in \mathbb{W}$, we have $(I-\Phi)^{-1}v^{\star}_k\in \mathbb{S}$. 
From the definition of $\hat{\mathbb{S}}^{\star}_k$ in \eqref{Eq:quantube}, it follows that $\hat{\mathbb{S}}^{\star}_k\subseteq \mathbb{S}$.
By the definitions of $h_s$ in \eqref{Eq:hs} and $h^{\star}_k$ in \eqref{Eq:hk}, it holds that $h_k^{\star}\leq h_s$. Let $\zeta_k=\max_{i}[\bm{1}-h_k^{\star}]_i/[\bm{1}-h_s]_{i}$. Then, if $h_s< \bm{1}$, we have $\zeta_k(\bm{1}-h_s)\geq (\bm{1}-h_k^{\star})\geq (\bm{1}-h_s)> \bm{0}$.   This implies that $\Omega\bigl(\zeta_k(\bm{1}-h_s),\nu\bigr)\supseteq\Omega(\bm{1}-h^{\star}_k,\nu)\supseteq \Omega(\bm{1}-h_s,\nu)$. Therefore, $\max_{z\in \Omega(1-h^{\star}_k,\nu)}\bar{F}\Psi^{\nu+1}z\leq \bm{1}-h_k^{\star}$ necessarily holds if $\nu$ satisfies $\max_{z\in \Omega(\zeta_k(\bm{1}-h_s),\nu)}\bar{F}\Psi^{\nu+1} z \leq 1-h_k^{\star}.$ But $\{z \mid z^TP_sz\leq 1\} \supseteq \Omega(\bm{1}-h_s,\nu_s)$, and hence $\{z \mid z^TP_sz\leq \zeta_k^2\} \supseteq \Omega(\zeta_k(\bm{1}-h_s),\nu_s)$. Furthermore, for any $\nu'\geq \nu$,  we have $\Omega(q,\nu) \supseteq \Omega(q,\nu')$. Thus it holds that $\{z \mid z^TP_sz\leq \zeta_k^2\} \supseteq \Omega(\zeta_k(\bm{1}-h_s),\nu_s)\supseteq  \Omega(\zeta_k(\bm{1}-h_s),\nu)$ for any $\nu\geq \nu_s$. This gives a pair of sufficient conditions to ensure $\max_{z\in \Omega(1-h^{\star}_k,\nu)}\bar{F}\Psi^{\nu+1}z\leq 1-h_k^{\star}$, that is, $\max_{z\in \{z \mid z^TP_sz\leq \zeta_k^2\}}\bar{F}\Psi^{\nu+1}z\leq \bm{1}-h^{\star}_k \text{ and } \nu\geq \nu_s.$ Equivalent conditions are  $[\bar{F}]_i\Psi^{\nu+1}P_s^{-1}\Psi^{\nu+1}[\bar{F}]^T_i\leq [1-h^{\star}_k]_i/\zeta^2_k$, $\forall i$,
and $\nu\geq\nu_s$.
\end{proof}

\vspace{-0.5cm}
\subsection{Recursive Feasibility}
From the statistical gap between  $\mathbb{W}_{\rm true}$ and $\hat{\mathbb{W}}^{\star}_k$ in Theorem~\ref{Theo: risk bound}, we have the following probabilistic recursive feasibility.    

\begin{proposition}
Suppose that at time step $k$, problem ${\rm OPT}(\hat{\mathbb{S}}^{\star}_k,h_k^{\star},\nu_k)$ is feasible. Then, with confidence no less than $1-\gamma$, problem ${\rm OPT}(\hat{\mathbb{S}}^{\star}_{k+1},h_{k+1}^{\star},\nu_{k+1})$ is feasible with probability at least $1-\epsilon$, where $\gamma$ and $\epsilon$ are defined in Theorem~\ref{Theo: risk bound}. 
\end{proposition}

\begin{proof}
Problem ${\rm OPT}(\hat{\mathbb{S}}^{\star}_{k+1},h_{k+1}^{\star},\nu_{k+1})$ is feasible if  $\hat{\mathbb{W}}^{\star}_{k+1}=\hat{\mathbb{W}}^{\star}_k$, which holds if and only if $w_k \in \hat{\mathbb{W}}^{\star}_k$.  The probabilistic recursive feasibility then follows from Theorem~\ref{Theo: risk bound}. 
\end{proof}
\vspace{-0.5cm}
\section{Case Study}\label{section:numerical}
\vspace{-0.3cm}
\begin{figure}[t]
\centering
\subfigure[\label{Fig: W_vs_W_true_50}]{\includegraphics[height=3.5cm]{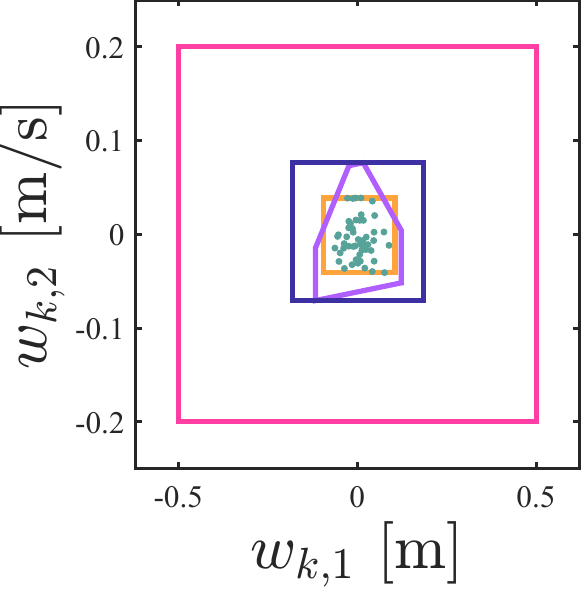}}
\subfigure[\label{Fig: F_vs_F_true_50}]{\includegraphics[height=3.5cm]{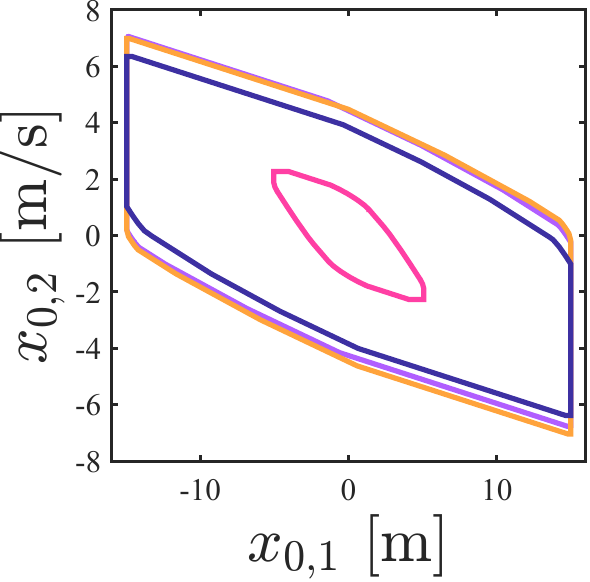}} 
\subfigure[\label{Fig: W_vs_W_true_20000}]{\includegraphics[height=3.5cm]{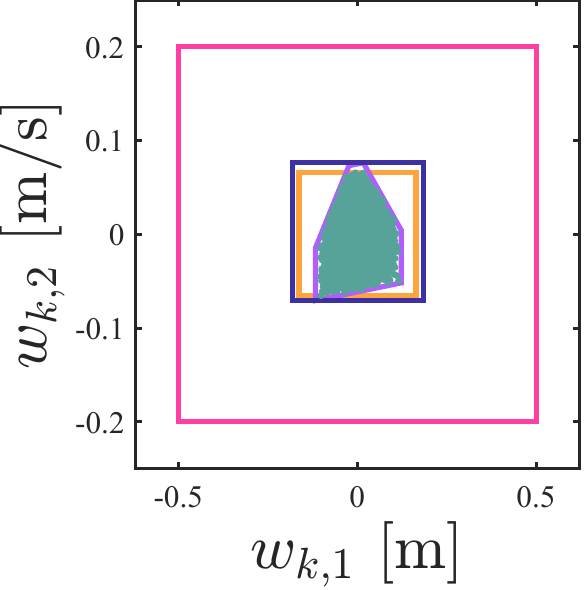}} 
\subfigure[\label{Fig: F_vs_F_true_20000}]{\includegraphics[height=3.5cm]{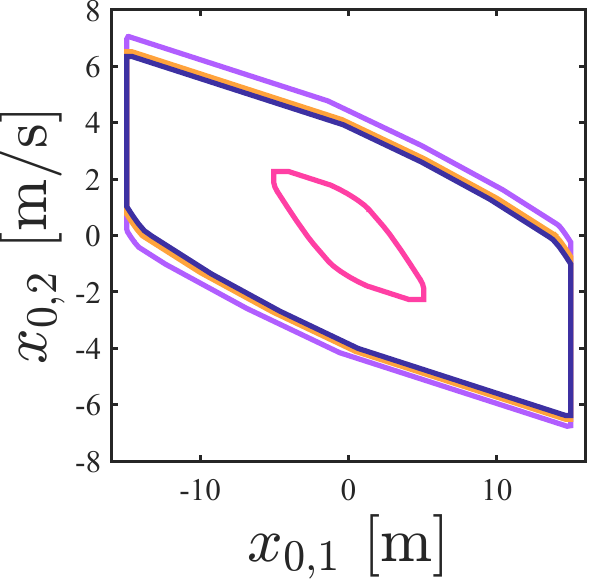}}
\caption{\small Comparison of disturbance sets  $\mathbb{W}$ [\textcolor{mycolorc1}{\rule[0.05cm]{0.75cm}{0.125em}}], $\mathbb{W}_{\rm true}$ [\textcolor{mycolorc3}{\rule[0.05cm]{0.75cm}{0.125em}}], $\hat{\mathbb{W}}_{\rm opt}$ [\textcolor{mycolorc4}{\rule[0.05cm]{0.75cm}{0.125em}}], and $\hat{\mathbb{W}}_0^{\star}$ [\textcolor{mycolorc2}{\rule[0.05cm]{0.75cm}{0.125em}}] and their corresponding feasible regions $\mathcal{F}_{\rm MPC}$ [\textcolor{mycolorc1}{\rule[0.05cm]{0.75cm}{0.125em}}], $\mathcal{F}_{\rm true}$ [\textcolor{mycolorc3}{\rule[0.05cm]{0.75cm}{0.125em}}], $\hat{\mathcal{F}}_{\rm opt}$ [\textcolor{mycolorc4}{\rule[0.05cm]{0.75cm}{0.125em}}], and $\hat{\mathcal{F}}_0$ [\textcolor{mycolorc2}{\rule[0.05cm]{0.75cm}{0.125em}}]. (a) The disturbance sets with $|\mathcal{I}_0^w| = 50$; (b) The feasible regions with $|\mathcal{I}_0^w| = 50$; (c) The disturbance sets with $|\mathcal{I}_0^w| = 20000$; (d) The feasible regions with $|\mathcal{I}_0^w| = 20000$.}
\label{Fig: F_vs_F_hat}
\end{figure}
\begin{figure}
\centering
\subfigure{\includegraphics[height=3.5cm]{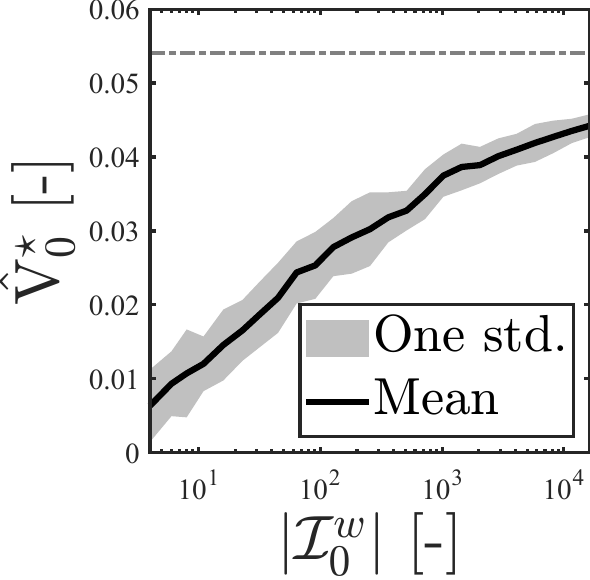}}
\subfigure{\includegraphics[height=3.5cm]{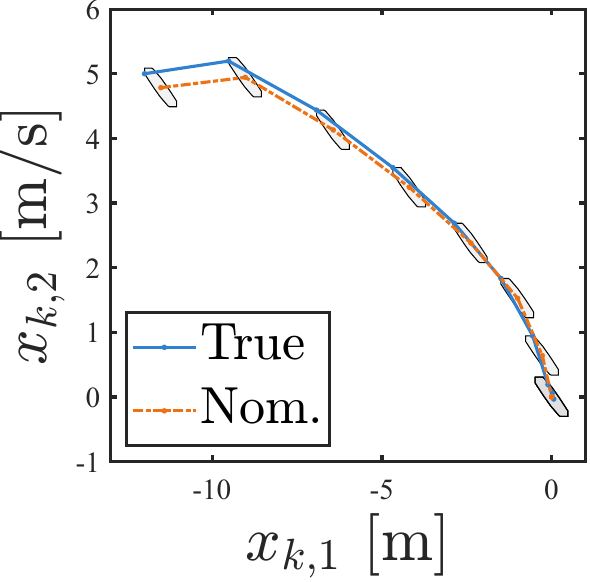}}
\subfigure{\includegraphics[height=3.5cm]{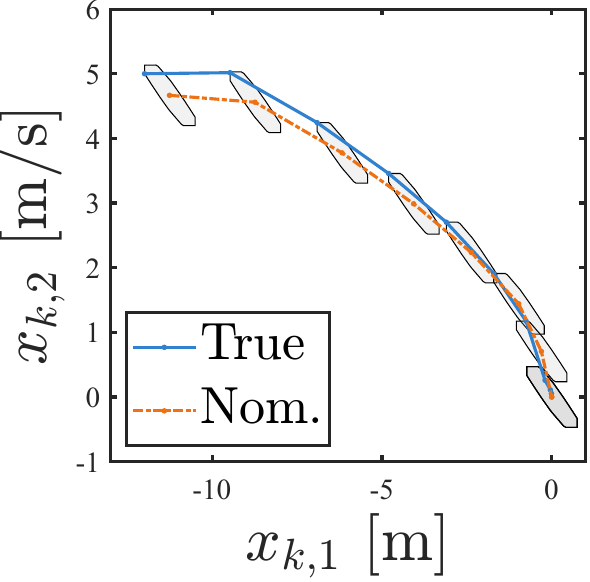}}
\subfigure{\includegraphics[height=3.5cm]{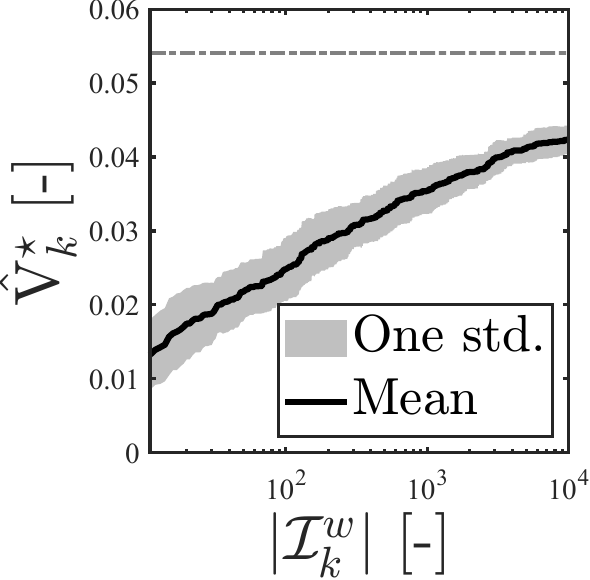}}
\caption{\small (a) Volume of $\hat{\mathbb{W}}_0^{\star}$ for MC simulations for different $\mathcal{I}^w_0$ (std. means standard deviation); (b) State trajectory by UQ-RMPC with $|\mathcal{I}_0^w|=100$; (c) State trajectory by UQ-RMPC with $|\mathcal{I}_0^w|=20000$; (d)Volume of $\hat{\mathbb{W}}_k^{\star}$ for MC simulations with $\mathcal{I}^w_k$.}
\label{Fig:onlinetraj}
\end{figure}
In this section, we consider a car-following example.  An ego autonomous vehicle (EV) tries to follow the longitudinal motion of a leading vehicle (LV) and keep the inter-vehicular distance as close to a pre-specified distance as possible. Dynamics of the EV and LV can be described by linear time-invariant models with superscripts $e$ and $l$ respectively, i.e., $
x_{k+1}^e = Ax_{k}^e + Bu_{k}^e + \xi_{k}^e, \
x_{k+1}^l = Ax_{i|k}^l + Bu_{k}^l+ \xi_{k}^l$, where $x_{k}^e = [p_{k}^e \ v_{k}^e]^{\rm T}$, $x_{k}^l = [p_{k}^l \ v_{k}^l]^{\rm T}$, $u_{k}^e=a_{k}^e \in \mathbb{R}$, $u_{k}^l=a_{k}^l \in \mathbb{R}$. The system and input matrices are $A = [\begin{smallmatrix} 1 & T\\ 0 & 1\end{smallmatrix}]$ and $B = [\begin{smallmatrix} 0 \\ T\end{smallmatrix}]$, and $T$ is a sampling interval. Here $p_{k}^e$, $p_{k}^l$ are the longitudinal positions of the EV and LV at time step $k$, respectively; $v_{k}^e$, $v_{k}^l$ are their longitudinal velocities; and $a_{k}^e$, $a_{k}^l$ are the accelerations. Variables $\xi_{k}^e \in \mathbb{R}^2$ and $\xi_{k}^l \in \mathbb{R}^2$ represent uncertainties of the EV and  LV models. We assume that  $\xi_{k}^l \in \Xi^l_{\rm true}$ and $\xi_{k}^e \in \Xi^e_{\rm true}$, where $\Xi^l_{\rm true}$ and $\Xi^e_{\rm true}$ are  the true uncertainty sets of the LV and the EV, respectively. In addition, we have $u_{k}^l \in \mathbb{U}^l_{\rm true}\subset \mathbb{R}$. Note that $\Xi^l_{\rm true}$, $\Xi^e_{\rm true}$,  and $\mathbb{U}^l_{\rm true}$  are unknown to the EV. Let $x^{\rm des} = [-L \ 0]^{\rm T}$, where $L$ is a desired safety distance. We define $x_{k} = x_{k}^e-x_{k}^l-x^{\rm des}$ and $u_k = a_k^e$. Thus, $w_{k} =  \xi_{k}^e-Bu_{k}^l-\xi_{k}^l$ is the disturbance of the  relative dynamics 
$x_{k+1} = Ax_{k} + Bu_{k} + w_{k}$
where  $w_{k} \in \mathbb{W}_{\rm true}:= \Xi^e_{\rm true} \oplus  (-\Xi^l_{\rm true}) \oplus (-B\mathbb{U}^l_{\rm true})$. 

The unknown set $\mathbb{W}_{\rm true}$ is overestimated by set $\mathbb{W}$. All the parameters for the implementations in this section are provided in our published code\footnote{\textcolor{blue}{\url{https://github.com/JianZhou1212/learning-based-rigid-tube-rmpc}}}. In the following, RMPC refers to the conventional robust MPC, i.e., the formulation \eqref{Eq:RigTubeMPCOPT}, while UQ-RMPC refers to the proposed algorithm, i.e., the formulation \eqref{Eq:QuanRigTubeMPCOPT}. Realisations from the sets $\Xi_{\rm true}^e$, $\Xi_{\rm true}^l$, and $\mathbb{U}_{\rm true}^l$ are uniformly selected at random. Problems \eqref{Eq:RigTubeMPCOPT} and \eqref{Eq:QuanRigTubeMPCOPT} are solved by \texttt{Ipopt}~\citep{wachter2006implementation} in \texttt{CasADi}~\citep{andersson2019casadi}, and set calculations are implemented by \texttt{Multi-Parametric Toolbox 3.0}~\citep{herceg2013multi} and \texttt{Yalmip} \citep{lofberg2004yalmip}. 

\vspace{-0.20cm}
\subsection{Comparison of the Feasible Regions} \label{Sec:Feasible Region Compare}
\vspace{-0.1cm}
We begin by comparing the feasible regions of RMPC and UQ-RMPC. For RMPC, the feasible region is defined as $\mathcal{F}(\mathbb{S} ,h_s,\nu_s)=\mathcal{F}(h_s,\nu_s) \oplus \mathbb{S}$, with $\mathcal{F}(h_s,\nu_s) = \Big\{s\in \mathbb{R}^{n_x}: \exists \textbf{c}  \text{ such that }   \bar{F}\Psi^i \begin{bmatrix} s \\ \bm{c} \end{bmatrix}\leq \bm{1}-h_s,   \forall i\in\mathbb{N}_{[0,\nu_s]}  \Big\}$. Similarly, the initial feasible region of UQ-RMPC is defined as $\mathcal{F}(\hat{\mathbb{S}}^{\star}_0 ,h^{\star}_0,\nu_0)$, which depends on the  initial disturbance information set $\mathcal{I}_0^w$.  As a baseline, we denote by $\hat{\mathbb{W}}_{\rm opt}$ the minimum set in form of \eqref{Eq:Wva} that covers $\mathbb{W}_{\rm true}$ and by $\hat{\mathcal{F}}_{\rm opt}$ its feasible region. In addition, the feasible region associated with $\mathbb{W}_{\rm true}$ is denoted by $\mathcal{F}_{\rm true}$. For simplicity, we write  $\mathcal{F}_{\rm MPC}=\mathcal{F}(\mathbb{S} ,h_s,\nu_s)$ and $\hat{\mathcal{F}}_{0}=\mathcal{F}(\hat{\mathbb{S}}^{\star}_0 ,h^{\star}_0,\nu_0)$. 

The comparison is shown in Fig.~\ref{Fig: F_vs_F_hat}. It is seen in Figs.~\ref{Fig: F_vs_F_hat}(a),(c) that the sets $\hat{\mathbb{W}}_0^{\star}$ are smaller than $\mathbb{W}$. This considerably increases the feasible region compared with $\mathcal{F}_{\rm MPC}$, as shown in Fig.~\ref{Fig: F_vs_F_hat}(b),(d). In addition, the set $\hat{\mathbb{W}}_0^{\star}$ approaches $\hat{\mathbb{W}}_{\rm opt}$ as $|\mathcal{I}_0^w|$ increases, and $\hat{\mathcal{F}}_0$ also approaches $\hat{\mathcal{F}}_{\rm opt}$. To further evaluate the impact of $\mathcal{I}_0^w$ on learning of $\hat{\mathbb{W}}_0^{\star}$, we compute the mean and standard deviation of the volume of $\hat{\mathbb{W}}_0^{\star}$, which is denoted by $\hat{\rm V}_0^{\star}$, by running the  Monte-Carlo (MC) simulations over $30$ different realisations, with  $|\mathcal{I}_0^w|$ ranging from $5$ to $10000$.  The results are shown in Fig.~\ref{Fig:onlinetraj}(a). We see that as $|\mathcal{I}_0^w|$ increases,  the mean of $\hat{\rm V}_0^{\star}$ approaches the volume of $\hat{\mathbb{W}}_{\rm opt}$ (the dashed line) and the standard deviation decreases. Despite the gap between $\hat{\mathbb{W}}^{\star}_0$ and $\hat{\mathbb{W}}_{\rm opt}$ shown in Fig.~\ref{Fig:onlinetraj}(a), the online implementation in Section~\ref{sec: Online Evaluation of UQ-RMPC} shows that the set $\hat{\mathbb{W}}^{\star}_0$  still exhibits sufficient robustness even if $|\mathcal{I}_0^w|$ is small.   

\vspace{-0.20cm}
\subsection{Online Evaluation of UQ-RMPC}\label{sec: Online Evaluation of UQ-RMPC}
\vspace{-0.1cm}
For online implementation, we run UQ-RMPC with $|\mathcal{I}_0^w|=100$ and $|\mathcal{I}_0^w|=20000$, respectively. 
The initial relative state is $x_0 = [-12 \  5]^{\rm T}$. The trajectories of the true relative state $x_k$ and the nominal state $s_{0|k}^{\star}$, and the associated sets $\hat{\mathbb{S}}^{\star}_k$, are shown in Fig.~\ref{Fig:onlinetraj}(b),(c), where both trajectories can be seen to converge to a neighbourhood of the origin. In addition, the set $\hat{\mathbb{S}}_k^{\star}$ increases with $|\mathcal{I}_0^w|$, which implies that the robustness is enhanced.

To investigate the online evolution of $\hat{\mathbb{W}}_k^{\star}$, we run MC simulations for $30$ realisations over $10000$ time steps with $|\mathcal{I}_0^w|=10$.  We compute the mean and standard deviation of volume of  $\hat{\mathbb{W}}_k^{\star}$, denoted by $\hat{\rm V}_k^{\star}$, during the iteration. The results are presented in Fig.~\ref{Fig:onlinetraj}(d). Similar to Fig.~\ref{Fig:onlinetraj}(a), the mean of $\hat{\rm V}_k^{\star}$ approaches the volume of $\hat{\mathbb{W}}_{\rm opt}$ (the dashed line) and the standard deviation decreases. We further evaluate robustness of UQ-RMPC by empirically quantifying the number of instances of infeasibility during implementation. We run  MC simulations for $300$ realisations with different $|\mathcal{I}_0^w|$. For each realisation, UQ-RMPC is run for $20$ time steps. We select the initial state $x_0$  close to the boundary of $\mathcal{F}(\hat{\mathbb{S}}_0^{\star}, h_0^{\star}, \nu_0)$, which is more likely to make the problem infeasible. The results are summarised in Table~\ref{Table_online_MC_MPC}, where a feasible realisation means the states and control inputs  satisfy the constraints, and $\epsilon_{\rm max}$ is calculated by $\epsilon_{\rm max} = \frac{{\rm e}}{{\rm e} - 1}({\rm ln}(\frac{1}{\gamma} + n_x))/|\mathcal{I}_0^w|$ in Theorem~\ref{Theo: risk bound} by fixing $\gamma = 0.005$. Here $\epsilon_{\rm max}$ indicates the risk of infeasibility. We see that UQ-RMPC encounters a risk of infeasibility when $|\mathcal{I}_0^w|$ is small, while the feasibility rate reaches $100\%$ when $|\mathcal{I}_0^w|$ increases. 
\begin{table}[t]
    \centering
    \scriptsize
    \caption{Feasibility of UQ-RMPC in random simulations.}
    \label{Table_online_MC_MPC}
    \begin{tabular}{cccccccc}
        \toprule
        $|\mathcal{I}_0^w|$ & $x_0$ & \textbf{Feas. Rate} & $\epsilon_{\rm max}$ & $|\mathcal{I}_0^w|$ & $x_0$ & \textbf{Feas. Rate} & $\epsilon_{\rm max}$ \\
        \midrule
        $10$ & $[-14.9 \ 6.841]^{\rm T}$ & $84.0\%$ & $0.8398$ & $100$ & $[-14.9 \ 6.491]^{\rm T}$ & $100\%$ & $0.0840$ \\
        $200$ & $[-14.9 \ 6.390]^{\rm T}$ & $100\%$ & $0.0420$ & $500$ & $[-14.9 \ 6.308]^{\rm T}$ & $100\%$ & $0.0168$ \\
        $2000$ & $[-14.9 \ 6.228]^{\rm T}$ & $100\%$ & $0.0042$ & $5000$ & $[-14.9 \ 6.232]^{\rm T}$ & $100\%$ & $0.0017$\\
        \bottomrule
    \end{tabular}
\end{table}

\vspace{-0.20cm}
\subsection{Comparison with Scenario MPC}
\vspace{-0.1cm}
We further compare the proposed UQ-RMPC with scenario MPC (SC-MPC) \citep{SCHILDBACH20143009}. The number of scenarios in SC-MPC is chosen as $K_{\rm sc} = |\mathcal{I}_0^w|/N$. The computations are performed on a standard laptop with an Intel i7-10750H CPU, 32.0 GB RAM running Ubuntu 22.04 LTS and MATLAB R2021b. Both MPC controllers are executed for $50$ steps, and the mean and the maximal computation time of solving UQ-RMPC and SC-MPC are summarised in Table.~\ref{Table_computation_time}. We see  that although both methods are convergent in these cases, UQ-RMPC significantly reduces the computation time compared with the SC-MPC. In addition, the computation time of UQ-RMPC generally remains stable when the sampling complexity increases, while for the SC-MPC, the computation time increases considerably with the sampling complexity.
\begin{table}[t]\scriptsize
\centering  
\begin{threeparttable}
\caption{Comparison of computation time between UQ-RMPC and SC-MPC.}  
\label{Table_computation_time}  
\begin{tabular}{ccccccccc}  
    \toprule  
    \multicolumn{2}{c}{\underline{\textbf {Sampling Complexity}}} & \multicolumn{3}{c}{\underline{\textbf {UQ-RMPC}}} & \multicolumn{3}{c}{\underline{\textbf {SC-MPC}}} \\
    \multicolumn{1}{c}{$|\mathcal{I}_0^w|$} & \multicolumn{1}{c}{$K_{\rm sc}$} & \multicolumn{1}{c}{\textbf{Mean}} & \multicolumn{1}{c}{\textbf{Max.}} & \multicolumn{1}{c}{\textbf{Convergence}} & \multicolumn{1}{c}{\textbf{Mean}} & \multicolumn{1}{c}{\textbf{Max.}} & \multicolumn{1}{c}{\textbf{Convergence}} \\
    \midrule
        $400$ & $50$ & $0.0674 \ {\rm s}$ & $0.155 \ {\rm s}$ & \checkmark & $0.141 \ {\rm s}$ & $0.233 \ {\rm s}$ & \checkmark \\ 
    $800$ & $100$ & $0.0641 \ {\rm s}$ & $0.0932 \ {\rm s}$ & \checkmark & $0.261 \ {\rm s}$ & $0.489 \ {\rm s}$ & \checkmark \\ 
        $1600$ & $200$ & $0.0640 \ {\rm s}$ & $0.0784 \ {\rm s}$ & \checkmark & $0.606 \ {\rm s}$ & $1.201 \ {\rm s}$ & \checkmark \\ 
        $3200$ & $400$ & $0.0941 \ {\rm s}$ & $0.114 \ {\rm s}$ & \checkmark & $2.273 \ {\rm s}$ & $3.522 \ {\rm s}$ & \checkmark \\ 
    \bottomrule
\end{tabular}
\end{threeparttable}
\end{table} 

\vspace{-0.3cm}
\section{Conclusion}\label{section:conclusion}
\vspace{-0.1cm}
This paper investigates robust tube MPC with online learning of uncertainty sets for discrete-time linear systems subject to state and input constraints. The observed disturbance realisations are used to learn the unknown true disturbance set by parameterising a prior given but conservative set. We provide a bound on the statistical gap between the true and quantified disturbance sets. The parameterisation of the quantified disturbance set allows the corresponding rigid tube bounding disturbance propagation to be computed efficiently. We propose an online implementation that updates the quantified disturbance set and the corresponding rigid tube at every time step. Numerical simulations  demonstrate the efficacy of our proposed algorithm and compare with conventional robust MPC and scenario MPC, respectively. Future directions of interest include the quantification of stochastic uncertainty and its integration with stochastic MPC, as well as investigation of scalability of this approach to higher dimensional systems.
\acks{\small This work is supported in part by Swedish Research Council Distinguished Professor Grant 2017-01078, Swedish Research Council International Postdoc Grant 2021-06727, Knut and Alice Wallenberg Foundation Wallenberg Scholar Grant, Deutsche Forschungsgemeinschaft (DFG, German Research Foundation) under Germany's Excellence Strategy - EXC 2075 – 390740016, and the Strategic Research Area at Link\"oping-Lund in Information Technology (ELLIIT).}
\bibliography{bibfile}

\end{sloppypar}
\end{document}